\documentclass{article}
\usepackage{amssymb,latexsym}
\usepackage{amsmath,amsfonts}
\usepackage{amsthm}
\usepackage{enumerate}
\usepackage{geometry}
\ifdefined\directlua
\usepackage{fontspec}
\usepackage{microtype}
\fi
\usepackage[british]{babel}

\newcommand{\cL}{\mathcal{L}}

\newcommand{\fP}{\mathfrak{P}}

\newcommand{\cC}{\mathcal{C}}

\newcommand{\cW}{\mathcal{W}}

\newcommand{\cG}{\mathcal{G}}

\newcommand{\PG}{\mathrm{PG}}

\newcommand{\F}{\mathbb{F}}
\newcommand{\FF}{\mathbb{F}}

\newcommand{\Rad}{\mathrm{Rad}}
\newcommand{\Res}{\mathrm{Res}}
\newcommand{\rank}{\mathrm{rank}\,}

\newcommand{\cP}{\mathcal{P}}
\newcommand{\bZ}{\bf{0}}
\newcommand{\codim}{\mathrm{codim}\,}
\theoremstyle{plain}
\newtheorem{lemma}{Lemma}[section]
\newtheorem*{mth}{Main Theorem}
\newtheorem{theorem}[lemma]{Theorem}
\newtheorem{corollary}[lemma]{Corollary}
\newtheorem{proposition}[lemma]{Proposition}
\newtheorem{assumption}[lemma]{Assumption}
\theoremstyle{definition}

\newcommand\sqval{\frac{q-1}{2}\left(q^{2n-(r+d)}-q^{2n-(r+d)-1}+q^{n-(r+d)/2}+
q^{n-(r+d)/2-1}\right)}

\def\nonsquare{\ensuremath{%
    \setbox0\hbox{$\square$}%
    \rlap{\hbox to \wd0{\hss\slash\hss}}\box0
}}

\title{Line Polar Grassmann Codes of Orthogonal Type}
\author{Ilaria Cardinali, Luca Giuzzi and Antonio Pasini}
\begin{document}
\maketitle
\begin{abstract}
Polar Grassmann codes of orthogonal type have been introduced in~\cite{IL13}. They are subcodes of the Grassmann code
arising from the projective system defined by the Pl\"{u}cker embedding of a polar Grassmannian of orthogonal type.
In the present paper we fully determine the minimum distance
of line polar Grassmann Codes of orthogonal type for $q$ odd.
\end{abstract}

\section{Introduction}
Codes $\cC_{m,k}$ arising from the Pl\"ucker embedding of the $k$--Grassmannians
of $m$--dimensional vector spaces have been
widely investigated since their first introduction in \cite{R1,R2}.
They are a remarkable generalization of Reed--Muller codes of the
first order and their
monomial automorphism groups and minimum weights are well understood,
see \cite{N96,GL2001,GPP2009,GK2013}.

Recently, in \cite{IL13}, the first two authors of the present paper introduced some new codes
$\cP_{n,k}$
arising from embeddings of orthogonal Grassmannians $\Delta_{n,k}$.
These codes correspond to the projective system
determined by the Pl\"ucker embedding of the
Grassmannian $\Delta_{n,k}$ representing all totally
singular $k$--spaces with respect to some non-degenerate quadratic
form $\eta$ defined on a vector space $V(2n+1,q)$
of dimension $2n+1$ over a finite field $\F_q$.
An orthogonal Grassmann code $\cP_{n,k}$ can be
obtained from the ordinary Grassmann code $\cC_{2n+1,k}$
by just deleting all the columns corresponding to $k$--spaces which
are non-singular with respect to $\eta$; it is thus a punctured
version of $\cC_{2n+1,k}$.
For $q$ odd, the
dimension of $\cP_{n,k}$ is the same as that of
$\cG_{2n+1,k}$, see \cite{IL13}.
The minimum  distance $d_{\min}$ of $\cP_{n,k}$ is always bounded away from $1$.
Actually, it has been shown in \cite{IL13} that for $q$ odd,
$d_{\min}\geq q^{k(n-k)+1}+q^{k(n-k)}-q$. By itself, this proves that the redundancy
of these codes is somehow better than that of $\cC_{2n+1,k}$.

In the present paper we prove the following theorem, fully determining
all the parameters for the case of line orthogonal Grassmann codes
(that is polar Grassmann codes with $k=2$) for $q$ odd.
\begin{mth}
\label{main1}
For $q$ odd,
  the minimum distance $d_{\min}$ of the orthogonal Grassmann code $\cP_{n,2}$
  is
  \[ d_{\min}=q^{4n-5}-q^{3n-4}. \]
  Furthermore, all words of minimum weight are projectively equivalent.
\end{mth}

Hence, we have the following.
\begin{corollary}
For $q$ odd, line polar Grassmann codes of orthogonal type are $[N,K,d_{\min}]$-projective codes with
\[N=\frac{(q^{2n-2}-1)(q^{2n}-1)}{(q^2-1)(q-1)},~~~K={{2n+1}\choose 2},~~~d_{\min}=q^{4n-5}-q^{3n-4}.\]
\end{corollary}

\subsection{Organization of the paper}
In Section \ref{S2} we recall some well--known facts on projective
systems and related codes, as well as the notion of polar Grassmannian
of orthogonal type. In Section \ref{S3} we  prove our
main theorem;
as some long, yet straightforward, computations are required,
here we present in full
detail the arguments for two of the main cases to be
considered, while we simply summarize the results, which can be obtained in an
analogous way, for the remaining two.
\section{Preliminaries}
\label{S2}
\subsection{Projective systems and Grassmann codes}
An $[N,K,d_{\min}]_q$ projective system $\Omega\subseteq\PG(K-1,q)$
is a set of $N$ points in $\PG(K-1,q)$
such that for any hyperplane $\Sigma$ of $\PG(K-1,q)$,
\[ |\Omega\setminus\Sigma|\geq d_{\min}. \]
Existence of $[N,K,d_{\min}]_q$ projective systems is equivalent to that of
projective linear codes with the same parameters; see,
for instance,~\cite{TVN}. Indeed, let $\Omega$ be a projective system
and denote by $G$ a matrix whose columns $G_1,\ldots,G_N$ are the
coordinates of representatives of the points of $\Omega$ with respect
to some fixed reference system.
Then, $G$ is the generator matrix of an $[N,K,d_{\min}]$ code over
$\FF_q$, say $\cC=\cC(\Omega)$. The code $\cC(\Omega)$
is not, in general, uniquely determined, but it is unique
up to code equivalence. We shall thus speak, with a slight
abuse of language, of \emph{the} code
defined by $\Omega$.

As any word $c$ of $\cC(\Omega)$ is of the form
$c=mG$ for some row vector $m\in\FF_q^K$,
it is straightforward to see that the number of
zeroes in $c$ is the same as the number of points of $\Omega$
lying on the hyperplane of equation
$m\cdot x=0$ where $m\cdot x=\sum_{i=1}^K m_ix_i$ and
$m=(m_i)_1^K$, $x=(x_i)_1^K$.
The minimum distance $d_{\min}$ of $\cC$ is thus
\begin{equation}
\label{eq-codes}
  d_{\min}=|\Omega|-f_{\max},~~~{\rm where}~~~ f_{\max}=\max_{\begin{subarray}{c}\Sigma\leq\PG(K-1,q)\\ \dim\Sigma=K-2 \end{subarray}}
|\Omega\cap\Sigma|.
\end{equation}
We point out that any projective code $\cC(\Omega)$ can also be
regarded, equivalently, as an evaluation code over $\Omega$ of degree $1$.
In particular, when $\Omega$ spans the whole of $\PG(K-1,q)=\PG(W)$, where $W$ is the underlying vector space, then
there is a bijection, induced by the standard inner product of $W$,
between the elements of the dual vector space $W^*$ and
the codewords $c$ of $\cC(\Omega)$.

Let $\cG_{2n+1,k}$ be the Grassmannian of the $k$--subspaces of a vector space
$V:=V(2n+1,q)$, with $k\leq n$ and let $\eta:V\to\FF_q$ be a non-degenerate
quadratic form over $V$.
Denote  by $\varepsilon_k:\cG_{2n+1,k}\to \PG(\bigwedge^k V)$
the usual Pl\"ucker embedding
\[
\varepsilon_k:\langle v_1,\ldots,v_k\rangle\to \langle v_1\wedge\cdots\wedge v_k\rangle.
\]

The orthogonal Grassmannian $\Delta_{n,k}$ is a geometry having
as points the $k$--subspaces of $V$ totally singular for $\eta$.
Let $\varepsilon_k(\mathcal{G}_{2n+1,k}):=\{\varepsilon_k(X_k)\colon X_k \text{ is a point of }\mathcal{G}_{2n+1,k}\}$ and
$\varepsilon_k(\Delta_{n,k})=\{\varepsilon_k(\bar{X}_k)\colon \bar{X}_k \text{ is a point of }\Delta_{n,k}\}.$
Clearly, we have
$\varepsilon_k(\Delta_{n,k})\subseteq\varepsilon_k(\cG_{2n+1,k})\subseteq\PG(\bigwedge^k V)$.
Throughout this paper we shall denote by
$\cP_{n,k}$  the code arising from the projective system
$\varepsilon_k(\Delta_{n,k})$.
By  \cite[Theorem 1.1]{CP2}, if
$n\geq 2$ and $k\in\{1,\ldots,n\}$, then
$\dim \langle{\varepsilon}_k(\Delta_{n,k})\rangle={{2n+1}\choose k}$ for $q$
odd,
 while
 $\dim \langle{\varepsilon}_k(\Delta_{n,k})\rangle={{2n+1}\choose k}-{{2n+1}\choose {k-2}}$ when $q$ is even.

We recall that for $k<n$, any line of $\Delta_{n,k}$
is also a line of $\cG_{2n+1,k}$.
For $k=n$, the lines of $\Delta_{n,n}$ are not lines of $\cG_{2n+1,n}$;
indeed,  in this
case $\varepsilon_n|_{\Delta_{n,n}}\colon \Delta_{n,n}\to\PG(\bigwedge^n V)$
 maps the lines of $\Delta_{n,n}$ onto non-singular conics of
 $\PG(\bigwedge^n V)$.

Thus, for $q$ odd,
the projective system identified by $\varepsilon_k(\Delta_{n,k})$
determines a code of length
$N=\prod_{i=0}^{k-1}\frac{q^{2(n-i)}-1}{q^{i+1}-1}$
  and dimension $K={{2n+1}\choose k}$; if $q$ is even
  the projective system identified by $\varepsilon_k(\Delta_{n,k})$
  determines instead a code of length
  $N=\prod_{i=0}^{k-1}\frac{q^{2(n-i)}-1}{q^{i+1}-1}$
  and dimension $K={{2n+1}\choose k}-{{2n+1}\choose {k-2}}.$

The following universal property provides
a well--known characterization of
alternating multilinear forms; see for instance \cite[Theorem 14.23]{ALA}.
\begin{theorem}
\label{ut}
Let $V$ and $U$ be vector spaces over the same field.
 A map $f: V^k \longrightarrow U$ is alternating
  $k$--linear if and only if there is a
  linear map $\overline{f}:
  \bigwedge^k V \longrightarrow U$ with $\overline{f}(v_1 \wedge v_2
  \wedge \cdots \wedge v_k) = f(v_1,v_2,\ldots,v_k)$.
  The map $\overline{f}$ is uniquely determined.
\end{theorem}
In general, the dual space $(\bigwedge^k V)^*$ of $\bigwedge^k V$ is isomorphic to the
space of all $k$--linear alternating forms of $V$. Observe that when
$\dim V=2n+1$, we can also write $(\bigwedge^k V)^*\cong\bigwedge^{2n+1-k}V$.

In this paper we are concerned with line Grassmannians, that is we
assume  $k=2$.
The above argument shows that for
any  hyperplane $\pi$ of $\PG(\bigwedge^2 V)$,
induced by a linear functional in $(\bigwedge^2 V)^*$, there is
an alternating bilinear form $\varphi_{\pi}:V\times V\to\FF_q$ such that
$p\wedge q\in\pi$ for $p,q\in V$ if and only if $\varphi_{\pi}(p,q)=0$.
In particular, when one considers the set of totally singular lines
of $V$ with respect to a given quadratic form $\eta$,
the image of a totally singular line $\ell=\langle p,q\rangle$ of $V$
belongs to the hyperplane $\pi$ if and only if $\ell$ is also totally isotropic
for $\varphi_{\pi}$, that is to say $\varphi_{\pi}(p,q)=0$.

Denote by $\mathcal{L}_{\varphi}$ the set of all totally isotropic lines
for the alternating form $\varphi:=\varphi_{\pi}$
corresponding to a hyperplane $\pi$ of $\PG(\bigwedge^2 V)$.
The number of points in $\varepsilon_2(\Delta_{n,2})\cap \pi$ is the same as
the number of lines of $\PG(V)$ simultaneously totally singular
for the quadratic form $\eta$ defining $\Delta_{n,2}$
 and totally isotropic for the alternating form $\varphi_{\pi}.$
Hence, by \eqref{eq-codes},
\label{mainc}
\[ d_{\min}(\cP_{n,2})=\#\{\text{points of } \Delta_{n,2}\}-
  \max_{\varphi} \#\{\text{points of } \Delta_{n,2}\cap\cL_{\varphi}\}. \]
In other words, in order to determine the minimum distance of $\cP_{n,2}$
 we need
to find the maximum number of lines which are simultaneously totally singular
for a fixed non-degenerate quadratic form $\eta$
on $V$ and totally isotropic for a
(necessarily degenerate) alternating form $\varphi$.

Recall that the \emph{radical} of $\varphi$ is the set
\[ \Rad(\varphi):=\{ v\in V: \forall w\in V, \varphi(v,w)=0 \}. \]
This is always a vector space and its codimension in $V$ is even.
As $\dim V$ is odd, $2n-1\geq \dim\Rad(\varphi)\geq 1$.

We point out that for the line projective Grassmann code $\cC_{2n+1,2}$,
it has been
proven in \cite{N96} that minimum weight codewords correspond to points
of  $\varepsilon_{2n-1}(\cG_{2n+1,2n-1})$; these can be regarded as
bilinear alternating forms of $V$ of maximum radical.
\par
In the case of orthogonal line Grassmannians, not all points
of $\cG_{2n+1,2n-1}$ yield codewords of $\cP_{n,2}$  of minimum weight.
However, as a consequence of the proof of our main result, we shall show in
Proposition \ref{mw} that
all the codewords of minimum weight of $\cP_{n,2}$
do indeed correspond to some $(2n-1)$--dimensional subspaces of $V$, that
is to say, to bilinear alternating forms of maximum radical.

\subsection{Generalities on quadrics}
\label{quadrics}
Let $Q:=Q(2t,q)$ be a non-singular
parabolic quadric of rank $t$ in $\PG(2t,q)$
and write $\kappa^0=(q^{2t}-1)/(q-1)$ for the number of
its points. The points of $\PG(2t,q)$ are
partitioned in three orbits under the action of the stabilizer
$\mathrm{PO}(2t+1,q)$ of $Q$ in $\mathrm{PGL}(2t+1,q)$; namely the points
of $Q$, those whose polar hyperplane cuts on $Q$  an elliptic
quadric $Q^-(2t-1,q)$ of rank $t-1$ and those whose polar
hyperplane  meets $Q$ in a hyperbolic quadric
$Q^+(2t-1,q)$ of rank $t$.
As customary, call the former points \emph{internal} and the
latter \emph{external to $Q$}.
Write $\kappa_0^-$ for the number of the internal points and $\kappa_0^+$
for that of the external ones.
Then, see e.g. \cite{HT91},
\[
\kappa_0^-=\frac{1}{2}q^t(q^t-1), \qquad
 \kappa_0^+=\frac{1}{2}q^t(q^t+1).\]

If $Q:=Q^+(2t-1,q)$ is a non-singular hyperbolic
quadric in $\PG(2t-1,q)$ or $Q:=Q^-(2t-1,q)$
is a non-singular elliptic  quadric in $\PG(2t-1,q)$,
then the polar hyperplane of a point $p$ not in $Q$ always cuts a
parabolic section of rank $t-1$
on $Q$. There are still
two orbits of $\mathrm{PO}(2t,q)$ on the non-singular points of
$\PG(2t-1,q)$; they have the same size, but can be distinguished by
the value (either square or non-square) assumed by the quadratic
form defining $Q$ on vectors representing their points.
Denote by $\kappa_+$
and $\kappa_-$ the size of these orbits,  in
the hyperbolic and elliptic case respectively.
 We have
\[ \kappa_+=\frac{1}{2}q^{t-1}(q^t-1),\qquad
\kappa_-=\frac{1}{2}q^{t-1}(q^t+1).
 \]
 Fix now a quadratic form $\eta$ on $V$ inducing a quadric $Q$
 and
 let the symbols $\square$ and $\nonsquare$
 stand  for the set of non-null square elements
 and  the set of non-square element of  $\F_q$.

With a slight abuse of notation, we shall say that
a (projective) point $p$ is  {\it square} and, consequently,
write $p\in\square,$
when $\eta(v_p)$ is a square for
any (non-null) vector $v_p$
representing the projective point $p=\langle v_p\rangle.$
Note that $\eta(v_p)\in \square $ if and only if
$\forall \lambda\in \F_q\setminus \{0\}$,
$\eta(\lambda v_p)\in \square$;
so, the above definition is well posed.
Analogously, we say that a (projective) point
$p$ is a {\it non-square} and  write $p\in
\nonsquare,$
when $\eta(v_p)$ is a non-square, for $v_p$  a (non-null) vector
representing $p=\langle v_p\rangle.$
Recall
the quadratic character of a point is constant on the
orbits of the orthogonal group.
In particular, in the parabolic case,
the external points are either all
squares or non-squares. For the internal points the opposite behavior holds.

\section{Proof of the Main Theorem}
\label{S3}
In order to simplify the notation, throughout this section, whenever no
ambiguity might arise, we shall usually denote by the same symbol
a point $p\in\PG(2n+1,q)$ and any non-null vector $v_p$ representing $p$
with respect to a suitably chosen basis. This slight lack of rigour
will however be harmless.

For $n=k=2$, by \cite[Main Result 2]{IL13}, the minimum distance of
the code $\cP_{2,2}$ is $d_{\min}=q^3-q^2$ and there is nothing to prove.
Suppose henceforth $n>2$, $k=2$ and $q$ odd.

As
\[ \#\{\text{points of } \Delta_{n,2}\}=\frac{(q^{2n}-1)(q^{2n-2}-1)}{(q-1)(q^2-1)}, \]
by Theorem \ref{mainc}, the part on the minimum distance in the Main Theorem
is equivalent to the following statement.
\begin{theorem}
\label{main1'}
Let $V:=V(2n+1,q)$, $q$ odd. The maximum number of lines totally singular
for a given non-singular quadratic form $\eta$ defined on $V$ and
simultaneously totally isotropic for a (degenerate) alternating form $\varphi$
over $V$ is
\[ f_{\max}=\frac{(q^{n-1}-1)}{(q-1)(q^2-1)}(q^{3n-2}+q^{3n-3}-q^{3n-4}+q^{2n}-q^{n-1}-1). \]
\end{theorem}
This section is fully
devoted to the proof of Theorem~\ref{main1'}.
In  Section \ref{tl} we shall introduce some preliminary lemmata.
Under a somehow further technical assumption we shall see that
four cases will need to be analyzed; they depend on the dimension and
position of the  radical $R$ of a generic alternating form $\varphi$ with
respect to the quadric $Q$.
In Section
\ref{fc} we shall perform a detailed analysis of the first of these
cases, the one yielding the actual minimum distance. The outcome of our
investigation for the next two cases
will be outlined in Section \ref{rc}; there we shall just describe
in what measure they differ from the case of Section \ref{fc} and
summarize the results obtained.
The fourth case will be dealt with in Section \ref{4c}.
In Section \ref{ga} we shall drop the assumption used in
sections \ref{fMS}--\ref{4c} and see that the theorem holds
in full generality.
Finally, in
Section \ref{pe} the projective equivalence of all the words
of minimum weight is proved.

\subsection{Some linear algebra}
\label{tl}
Throughout the remainder of the paper, we shall always denote by
$\eta$ a fixed non-singular quadratic form on $V$ and by
$\varphi$ an arbitrary alternating form defined on the same space.
We shall also write $M$ and $S$ for the matrices representing
respectively $\eta$ and $\varphi$ with respect to a given, suitably chosen,
basis $B$ of $V$;
write also
$\perp_Q$ for the orthogonal polarity induced by $\eta$ and
$\perp_W$ for the (degenerate) symplectic polarity induced by $\varphi$.
In particular, for $v\in V$, the symbols
$v^{\perp_Q}$ and $v^{\perp_W}$ will respectively
denote the space orthogonal to $v$ with respect to $\eta$ and $\varphi.$
Likewise, when $X$ is a subspace of $V$, the notations
$X^{\perp_Q}$ and $X^{\perp_W}$ will be used to denote
 the spaces orthogonal to $X$ with respect to $\eta$ and $\varphi.$
We shall say that a subspace $X$ is \emph{totally singular} if
$X\leq X^{\perp_Q}$ and \emph{totally isotropic} if $X\leq X^{\perp_W}$.
Let also $R:=\Rad(\varphi)$ and $r:=\dim R$.
\begin{lemma}\label{prop1}
\begin{enumerate}
\item For any $v\in V$,  $v^{\perp_Q}=v^{\perp_W}$ if and only if
  $v$ is an eigenvector of non-zero eigenvalue of $M^{-1}S$.
\item The radical $R$ of
$\varphi$ corresponds to the eigenspace
of $M^{-1}S$ of eigenvalue $0$.
\end{enumerate}
\end{lemma}
\begin{proof}
\begin{enumerate}
\item
Observe that
$v^{\perp_Q}= v^{\perp_W}$ if and only if the equations
$x^TMv=0$ and $x^TSv=0$ are equivalent for any $x\in V$.
This means that
there exists an element $\lambda\in\FF_q\setminus\{0\}$
such that $Sv=\lambda Mv.$
As $M$ is non-singular, the latter says that $v$ is an eigenvector of
 non-zero eigenvalue $\lambda$ for $M^{-1}S$ .
\item
Let $v$ be an eigenvector of $M^{-1}S$ of eigenvalue $0.$
Then $M^{-1}Sv=0$, hence $Sv=0$ and $x^TSv=0$
for every $x\in V$, that is $v^{\perp_W}=V$. This means $v\in R$.
\end{enumerate}
\vskip-.3cm
\end{proof}
We can now characterize the eigenspaces of $M^{-1}S$.
\begin{lemma}\label{prop2}
Let $\mu$ be a non-zero eigenvalue of $M^{-1}S$ and $V_{\mu}$
be the corresponding eigenspace. Then,
\begin{enumerate}
\item\label{pp1} $\forall v\in V_{\mu}$ and  $r\in R$,  $r\perp_Q v.$
  Hence, $V_{\mu}\leq R^{\perp_Q}$.
\item
  The eigenspace $V_{\mu}$ is
  both totally isotropic for $\varphi$ and totally singular for $\eta$.
\item\label{pp3} Let $\lambda,\mu\neq0$ be two not necessarily
  distinct eigenvalues of $M^{-1}S$ and $u$, $v$ be two
  corresponding eigenvectors.
  Then either of the following  holds:
  \begin{enumerate}
  \item $u\perp_Q v$ and $u\perp_W v.$
  \item $\mu=-\lambda.$
  \end{enumerate}
\end{enumerate}
\end{lemma}
\begin{proof}
\begin{enumerate}
\item
  Take $v\in V_{\mu}$. As $M^{-1}Sv=\mu v$ we also have
  $\mu v^T=v^TS^TM^{-T}.$ So, $v^TM^T=\mu^{-1}v^TS^T.$
Let $r\in R$. Then, as $S^T=-S$, $v^TMr=\mu^{-1}v^TS^Tr$ and $v^TSr=0$ for any $v$, we have $v^TMr=0$, that is
$r\perp_Q v$.
\item Let $v\in V_{\mu}$. Then $M^{-1}Sv=\mu v$, which implies $Sv=\mu Mv.$
  Hence, $v^TSv=\mu v^TMv.$
  Since $v^TSv=0$ and $\mu\not=0$, we also have $v^TMv=0$,
  for every $v\in V_{\mu}$.
  Thus, $V_{\mu}$ is totally singular for $\eta$.
  Since $V_{\mu}$ is totally singular, for any $u\in V_{\mu}$
  we have $u^TMv=0$; so, $u^TSv=\mu u^T Mv=0$,
  that is $V_{\mu}$ is also totally isotropic.
\item
Suppose that either $u\not \perp_Q v$ or  $u\not \perp_W v.$
Since by Lemma~\ref{prop1}
$u^{\perp_Q}= u^{\perp_W}$
 and $v^{\perp_Q}= v^{\perp_W}$,
we have $Mu=\lambda^{-1}Su$ and $Mv=\mu^{-1}Sv$.
So, $u\not \perp_Q v$ or  $u\not \perp_W v$ implies $v^TMu\neq 0\neq v^TSu$.
Since $M^{-1}Su=\lambda u$ and  $M^{-1}Sv=\mu v$, we have
\[ v^TSu=v^TS(\lambda^{-1}M^{-1}Su)=
\lambda^{-1}(-M^{-1}Sv)^TSu=-(\lambda^{-1}\mu)v^TSu; \]
 hence, $-\lambda^{-1}\mu=1.$
\end{enumerate}
\vskip-.3cm
\end{proof}
\begin{corollary}\label{cor5}
  Let $V_{\lambda}$ and $V_{\mu}$ be two eigenspaces of non-zero
  eigenvalues $\lambda\neq-\mu$.
  Then, $V_{\lambda}\oplus V_{\mu}$ is both totally singular and totally isotropic.
\end{corollary}
\begin{proposition}\label{prop4}
If $x\in V_{\lambda}$, then $\forall y\in V,
\lambda y^T M x=y^T S x.$
\end{proposition}
\begin{proof}
If $x\in V_{\lambda}$, then $M^{-1} S x=\lambda x$; hence, $\forall y\in V$,
 $y^TSx=\lambda y^T Mx$.
\end{proof}

\begin{lemma}
\label{maxeig}
The maximum number of eigenvectors for $M^{-1}S$ of non-zero eigenvalue
is obtained when a complement $H_0$ of  $R\cap R^{\perp_Q}$ in $R^{\perp_Q}$
contains a direct sum $V_{\mu}\oplus V_{\lambda}$
of two eigenspaces of $M^{-1}S$, each of dimension $m$, where
$m$ is the rank of the non-singular quadric $Q_0:=Q\cap H_0$.
\end{lemma}
\begin{proof}
By Claim 2 of Lemma~\ref{prop2}, any maximal eigenspace
$V_{\mu}$ of $M^{-1}S$ with
non-zero eigenvalue is both totally singular for $\eta$
and totally isotropic for $\varphi$.  By Claim 1 of Lemma \ref{prop2},
$V_{\mu}$ is contained in a complement $H_0$ of $R\cap R^{\perp_Q}$ in $R^{\perp_Q}$.
In particular, $V_{\mu}$ is contained in a generator of the
quadric $Q_0,$ so $\dim V_{\mu}\leq m$.
If there were at least three distinct eigenspaces
$V_{\lambda},V_{\mu},V_{\theta}$ with $\lambda=-\mu$, then, obviously,
$\theta\neq\pm\lambda, \pm\mu$. Let $c=\dim V_{\theta}\geq 1$.
By Corollary \ref{cor5}, both
$V_{\theta}\oplus V_{\lambda}$ and $V_{\theta}\oplus V_{\mu}$
are totally singular for $\eta$; hence they are contained in
two generators, say $G_+$ and $G_-$ of $Q_0$,
with $V_{\theta}\leq G_+\cap G_-$ and $c<m$,
$\dim V_{\lambda}, V_{\mu}\leq m-c$. Thus, we have the following upper bond
on the number of eigenvectors of non-zero eigenvalue:
\[ |V_{\lambda}|+|V_{\theta}|+|V_{\mu}|-3\leq 2q^{m-c}+q^c-3<2q^m-2=|G_+|+|G_-|-2. \]
This is to say that the possible maximum number of eigenvectors of
non-zero eigenvalue attained when there
are at least three distinct non-zero eigenvalues is
strictly less than the number of
vectors contained in two vector spaces of dimension $m$.

We now show that there actually are
alternating forms
$\varphi$ inducing two eigenspaces of dimension $m$;
this yields that the number
$2(q^m-1)$ of eigenvectors can be achieved and, consequently, this is the
maximum possible.
Let $G_+$ and $G_-$ be two trivially intersecting generators of $Q_0$ with
bases respectively $\{b_i^+\}_{i=1}^m$ and $\{b_i^-\}_{i=1}^m$.
We can suppose without loss of generality that the quadratic form $\eta|_{V'},$ restriction of $\eta$ to $V':=G_+\oplus G_-$, is represented
with respect to the basis $B'=\{b_i^+\}_{i=1}^m\cup\{b_i^-\}_{i=1}^m$ by the
matrix
\[ M'=\begin{pmatrix}
  \bZ_m & I_m \\
  I_m & \bZ_m
\end{pmatrix}, \]
where $I_m$ is the $m\times m$ identity matrix and $\bZ_m$
stands for the null matrix of order $m.$ Choose also $\eta$ such that
$V$ decomposes as $V=V_0\perp_QV'\perp_Q R$,
where $V_0$ is an orthogonal complement of $V'\oplus R$ with respect to
$\perp_Q$.
Define now an alternating form $\varphi$ with radical $R$
represented  on $V'$ with respect to
$B'$ by the matrix
\[ S'=\begin{pmatrix}
  \bZ_m & -I_m \\
  I_m & \bZ_m \\
 \end{pmatrix} \]
and such that we also have $V=V'\perp_WV_0\perp_W R$.
This is always
possible, as $\dim (V'+V_0)$ is even.
For any  $v\in G_+\cup G_-\subseteq V'$,
\[ v^{\perp_W}=v^{\perp_W'}+R+V_0=v^{\perp_Q'}+R+V_0=v^{\perp_Q}, \]
where by $\perp_W'$ and $\perp_Q'$ we denote the orthogonality
relations defined by the
restriction of the forms $\eta$ and $\varphi$ to respectively
$V'$ and $V'\times V'$. By Lemma \ref{prop1},  $v$ is an
eigenvector of $M^{-1}S$. Thus,
$G_+$, $G_-$ are eigenspaces of $M^{-1}S$
of dimension $m$.
\end{proof}

By Lemma \ref{maxeig}, the
alternating forms $\varphi$ inducing a maximum
number of eigenvectors of $M^{-1}S$, determine
two eigenspaces $V_{\lambda}$ and $V_{\mu}$
with $\dim V_{\lambda}=\dim V_{\mu}=m$. In this case,
Lemma \ref{prop2}, point \ref{pp3} shows that, for
$V_{\lambda}$ and $V_{\mu}$ to be both maximal,
$\lambda=-\mu$ is also required.

\subsection{Sketch of the proof and setup}
\label{sketch}

As outlined before our aim is to count the maximum
number $f$ of lines totally isotropic for $\varphi$
and totally singular for $\eta.$

Let $p$ be a point of $Q$ and consider the spaces $p^{\perp_Q}$ and
$p^{\perp_W}$.  Since $Q$ is non-singular, $p^{\perp_Q}$ is a
hyperplane of $\PG(V)$ for any  $p\in Q$, while $p^{\perp_W}$
is a hyperplane of $\PG(V)$ if and only if $p\not\in R.$

Let now $Q_p$ be the orthogonal geometry induced by  $\eta$ on
$p^{\perp_W}$ and denote by $\Res_{Q_p}p$ the geometry having as
elements the (singular) subspaces (with respect to $\eta$) through
$p$ contained in $Q_p.$

As each line contains
$q+1$ points and each line through $p$ in $p^{\perp_W}\cap p^{\perp_Q}$
corresponds to a point in $\Res_{Q_p}p$, the number of lines
simultaneously totally isotropic for $\varphi$  and totally singular
for $\eta$ is
\begin{equation}
\label{main f}
f=\frac{1}{q+1}\sum_{p\in Q} \tau(p),
\text{ where }
\tau(p):=\#\{ \text{points of } \Res_{Q_p} p \}.
\end{equation}
We distinguish two main cases.
\begin{itemize}
\item
{\bf{Case A: $p^{\perp_Q}\subseteq p^{\perp_W}$}}\\
Let $\fP:=\fP_a\cup \fP_b$ and $A:=|\fP|$, where
\[\fP_a:=\{p\in Q\colon p^{\perp_W}=\PG(V)\},\,A_R:=|\fP_a| \,\,\text{ and }\]
\[\fP_b:=\{p\in Q\colon p^{\perp_W}=p^{\perp_Q}\} ,\,A_V:=|\fP_b|.\]
 For any $p\in\fP$,  $\Res_{Q_p}p\cong Q(2n-2,q)$ (where $Q(2n-2,q)$ is a non-singular parabolic quadric of rank $n-1$).
Thus, we have
  \[ p\in\fP \Rightarrow \tau(p)=\frac{q^{2n-2}-1}{q-1}=:A^0. \]
  The points
  in $\fP_a$ are the points of $Q$ contained in $R$;
  by Lemma~\ref{prop1} the points in $\fP_b$ correspond to
  eigenvectors of $M^{-1}S$   of non-zero eigenvalue. In particular,
  \[ A_{R}=\#\{ \text{points of } Q\cap R\},\qquad
  A_{V}=\frac{|\bigcup_{\lambda\neq 0} V_{\lambda}|-1}{q-1} \]
  where $V_{\lambda}$ are the eigenspaces of $M^{-1}S$ as
  $\lambda$ varies among all of its non-null eigenvalues.
Clearly, $A=A_R+A_V.$
\item
{\bf{Case B: $\codim_{p^{\perp_W}} p^{\perp_W}\cap p^{\perp_Q}=1$}}\\
Three possibilities can occur for $\Res_{Q_p}p$:
 \begin{enumerate}
 \item   $\Res_{Q_p}p\cong Q^+(2n-3,q)$ is a non-singular hyperbolic quadric of rank $n-1$ in the $(2n-3)$-dimensional projective space $p^{\perp_Q}\cap p^{\perp_W}$; let
   \[\fP^+:=\{p\in Q\colon \Res_{Q_p}p\cong Q^+(2n-3,q)\}\,\,\text{and}\,\, N^+=|\fP^+|. \]
   In particular,
    \[ p\in\fP^+ \Rightarrow \tau(p)=\frac{(q^{n-1}-1)(q^{n-2}+1)}{q-1}=:B^+. \]
  \item  $\Res_{Q_p}p\cong Q^-(2n-3,q)$ is a non-singular elliptic quadric of rank $n-2$ in the $(2n-3)$-dimensional projective space $p^{\perp_Q}\cap p^{\perp_W}$; define
    \[\fP^-:=\{p\in Q\colon \Res_{Q_p}p\cong Q^-(2n-3,q)\}\,\,\text{and}\,\, N^-=|\fP^-|.\]
    Then,
    \[ p\in\fP^- \Rightarrow \tau(p)=\frac{(q^{n-1}+1)(q^{n-2}-1)}{q-1}=:B^-. \]
  \item $\Res_{Q_p}p$ is isomorphic to a cone $TQ(2n-4,q)$ having a point $T$ as vertex and a non-singular parabolic quadric
    $Q(2n-4,q)$ of rank $n-2$ as base; put
    \[\fP^0:=\{p\in Q\colon \Res_{Q_p}p\cong TQ(2n-4,q)\}\,\,\text{and}\,\, N^0=|\fP^0|.\]
    Then,
    \[ p\in\fP^0 \Rightarrow \tau(p)=\frac{(q^{2n-3}-1)}{q-1}=:B^0, \]
  \end{enumerate}
\end{itemize}
Clearly, as pointsets, $Q\setminus \fP=\fP^+\cup \fP^0\cup \fP^-.$

By replacing the aforementioned numbers in \eqref{main f}, we
obtain
\begin{equation}
\label{key}
 f=\frac{1}{q+1}(AA^0+N^0B^0+N^+B^++N^-B^-).
\end{equation}

The aim of the remainder of the current paper is to determine
the quantities $A,N^0,N^+$ and $N^-$ in such a way as to compute
the maximum $f_{\max}$ of  $f$.

Write $Q_R:=R\cap Q$ for
the quadric induced by $\eta$ on $R$ and take $D$ as
the radical of $Q_R$; this is to say $D=\Rad(\eta|_{R})$; write also
$d=\dim D$. Observe that, in general, $R\leq D^{\perp_Q}$ and
the space $V$ decomposes  as follows
\[ V=H\oplus D^{\perp_Q};\qquad D^{\perp_Q}=H_0\oplus R;\qquad R=D_0\oplus D, \]
where $D_0$ is a direct complement of $D$ in $R,$ $H$ is a direct
complement of $D^{\perp_Q}$ in $V$ and $H_0$ is a direct complement of $R$ in $D^{\perp_Q}.$ Thus,
\[ V=H\oplus H_0\oplus D_0\oplus D. \]
Let also \[ R_0:=Q\cap D_0,\qquad Q_0:=Q\cap H_0  \]
be the quadrics induced by $\eta$ in respectively $D_0$ and $H_0$.
As $Q$ is non-singular we  have
\[ \dim H=\dim D=d,\qquad \dim H_0=2n+1-(r+d),\qquad \dim D_0=r-d. \]
Denote by $m$
the rank of $Q_0$;
since for any generator $X$ of $Q_0$ we have
$X+D\subseteq Q$, then
$d+m \leq n$.
The function $f$ in~\eqref{key} is then dependent on $r$ and $d.$
The possible ranks of $R_0$ and $Q_0$ are outlined in Table
\ref{table cases}. These correspond to four cases to investigate.
In particular, we shall denote by $f^i(r,d)$, $1\leq i\leq 4$, the function
providing
the values of $f$ in a given case $i$ and by $f^i_{\max}$ its corresponding
maximum.

\subsection{Forms for $M$ and $S$}
\label{fMS}
In this section we shall determine suitable forms for the matrix $M$
and $S$ which should provide the maximum possible values for $f$
under the following Assumption \ref{eigA}; this shall be silently used
in Sections \ref{fc}--\ref{4c} and removed in Section \ref{ga}.
\begin{assumption}\label{eigA}
The maximum of the function $f$ can be attained only if for a given
radical $R$ the number of eigenvectors $A_V$ for  $M^{-1}S$
is maximum.
\end{assumption}

It is always possible to fix an ordered basis
$B=B_H\cup B_{H_0}\cup B_{D_0}\cup B_D$ of $V$ such that
\begin{equation}
\label{basis}
\begin{array}{l}
B_H=\{b_1\ldots b_{d}\}{\rm{~is~ an~ ordered~ basis~ of~}} H;\\
B_{H_0}=\{b_{d+1}\ldots b_{2n+1-r}\} {\rm{~is~ an~ ordered~ basis~ of~}} H_0;\\
B_{D_0}=\{b_{2n+2-r},\ldots, b_{2n+1-d}\} {\rm{~is~ an~ ordered~ basis~ of~}} D_0;\\
B_{D}=\{b_{2n+2-d}\ldots b_{2n+1}\} {\rm{~is~ an~ ordered~ basis~ of~}} D.\\
\end{array}
\end{equation}

As all parabolic quadrics
of given rank are projectively equivalent,
the matrix $M$ representing $Q$ with respect to $B$  may
be taken of the form
\begin{equation}
\label{mmat}
 M=\begin{pmatrix}
  \bZ & \bZ & \bZ & I_d \\
  \bZ &Q_0& \bZ & \bZ   \\
  \bZ & \bZ &R_0& \bZ \\
  I_d& \bZ & \bZ & \bZ
  \end{pmatrix},
\end{equation}
where $R_0$ and $Q_0$ are given by
Table~\ref{table matrices}, according to the cases of
Table~\ref{table cases}.
Here, with a slight abuse of notation, as no ambiguity may arise,
we use $R_0$ and $Q_0$ to denote
both the matrices and the corresponding quadrics.

\begin{table}
\centerline{
\begin{tabular}{c|c|c|c|c|c}
  case & parity of $d$ & type of $R_0$  & $\rank R_0$ & type of $Q_0$ &
  $\rank Q_0$ \\ \hline
   1   &    odd        & hyperbolic     & $(r-d)/2$ & parabolic & $n-(r+d)/2$ \\
   2   &    odd        & elliptic       & $(r-d)/2-1$ & parabolic
   & $n-(r+d)/2$ \\
   3   &   even        & parabolic      & $(r-d-1)/2$ & hyperbolic
   & $n-(r+d-1)/2$ \\
   4   &   even        & parabolic      & $(r-d-1)/2$ & elliptic
   & $n-(r+d+1)/2$ \\
\end{tabular}
}
\caption{Decomposition of the quadric $Q$}
\label{table cases}
\end{table}

\begin{table}
\begin{small}
  \centerline{
    \begin{tabular}{c|c|c}
      Case & $R_0$ & $Q_0$ \\ \hline
       1   & $R_0^+:=\begin{pmatrix} \bZ & I \\ I & \bZ
                 \end{pmatrix}$
           & $\begin{pmatrix} \bZ & I & \bZ \\
                         I & \bZ & \bZ \\
                         \bZ & \bZ & 1 \\
                  \end{pmatrix}$ \\
       2   & $R_0^-:=\begin{pmatrix} \bZ & I & \bZ & \bZ \\
         I & \bZ & \bZ & \bZ \\
         \bZ    &  \bZ     & 1 & 0 \\
         \bZ    &  \bZ     & 0 & -\xi
       \end{pmatrix}$
       & $\begin{pmatrix} \bZ & I & \bZ \\
         I & \bZ &  \bZ \\
         \bZ & \bZ & 1 \\
       \end{pmatrix}$ \\
     \end{tabular}\qquad
    \begin{tabular}{c|c|c}
      Case & $R_0$ & $Q_0$ \\ \hline
     3   & $\begin{pmatrix} \bZ & I & \bZ \\
                              I & \bZ & \bZ  \\
                              \bZ    &  \bZ     & 1
                              \end{pmatrix}$
                  & $Q_0^+:=\begin{pmatrix} \bZ & I  \\
                         I & \bZ  \\
                   \end{pmatrix} $ \\
     4    & $\begin{pmatrix} \bZ & I & \bZ \\
                              I & \bZ & \bZ  \\
                              \bZ    &  \bZ     & 1
                              \end{pmatrix}$
                  & $Q_0^-:=\begin{pmatrix} \bZ & I & \bZ & \bZ \\
                         I & \bZ & \bZ & \bZ \\
                         \bZ & \bZ & 1 & 0 \\
                         \bZ & \bZ & 0 & -\xi
                  \end{pmatrix}$ \\
        \end{tabular}
}
\end{small}
\caption{Matrices for $R_0$ and $Q_0$}
\label{table matrices}
In this table $\xi$ is a non-square in $\FF_q$; the order of the identity matrices
$I$ and the null matrices $\bZ$ is provided by Table \ref{table cases}.
\end{table}
Observe that
there is always a vector $x=(0,\ldots,0,1,0\ldots 0)\in V$
such that $x^TMx=1$; the exact $x$ to be chosen according varies to the case
being considered as described
in Table \ref{intext}.
It can
be seen directly that the hyperplane $x^{\perp_Q}$ cuts $Q$ in a
section which
is hyperbolic for cases 1 and 3 and elliptic otherwise; thus,
the correspondence
of Table \ref{intext} between square/non-square points and internal/external
points to $Q$ is determined  using
the remarks of Section \ref{quadrics}.

\begin{table}[h]
\centerline{
  \begin{tabular}{c|c|c|c|c}
    case & $x$ & $x^{\perp_Q}\cap Q$  & internal points & external points \\ \hline
     1   &  $(0_{2n-r},1,0_r)$ & hyperbolic &  non-square     &   square     \\
     2   &  $(0_{2n-r},1,0_r)$ & elliptic &   square      &    non-square \\
     3   &  $(0_{2n-d},1,0_d)$ & hyperbolic & non-square     &   square     \\
     4   &  $(0_{2n-d},1,0_d)$ & elliptic &   square      &    non-square \\
     \end{tabular}
}
\caption{Internal/external points for $Q$}
\label{intext}
\end{table}

\begin{lemma}
\label{lintext}
Let $p\not\in R$ be a singular point with respect to $\eta.$
If the points external to $Q$ are squares, then
\begin{itemize}
\item
  $p\in\fP^+$ if and only if $p^TSM^{-1}Sp\in-\square$;
\item $p\in\fP^0$ if and only if $p^TSM^{-1}Sp=0$;
\item $p\in\fP^-$ if and only if $p^TSM^{-1}Sp\in-\nonsquare$.
\end{itemize}
When the points external to $Q$ are non-squares, the classes
$\fP^+$ and $\fP^-$ are exchanged.
\end{lemma}
\begin{proof}
Let $W_p:=p^{\perp_W}$ and write
$a_p:=W_p^{\perp_Q}$ for the point orthogonal with respect to
$\eta$ to the hyperplane
$W_p$.  Then, $a_p^{\perp_Q}=p^{\perp_W}$.
In particular, the following two equations are equivalent for any $x$:
\[x^TSp=0,\qquad x^TMa_p=0.\]
In other words, there exists $\rho \in \FF_q\setminus\{0\}$ such
that $\rho Sp= Ma_p$ and, consequently,
$a_p=\rho M^{-1}Sp$. Observe that
the point $a_p$ belongs to $p^{\perp_Q}$, as
$p^TMa_p=\rho p^TMM^{-1}Sp=\rho p^TSp=0$.
Clearly, $p$ is an eigenvector of $M^{-1}S$ if and only if $p=a_p$.
In this case $p\in\fP$.

Suppose now $p$ not to be an eigenvector of $M^{-1}S$ and consider
the quadric
\[ Q_p=W_p\cap Q=a_{p}^{\perp_Q}\cap Q. \]
Observe that the  the residue at $p$ of $Q_p$
is either an hyperbolic, elliptic or degenerate quadric (more precisely, in the latter case, a cone with vertex a point and base a parabolic quadric) according as $a_p$ is
external, internal or contained in $Q_p\cap p^{\perp_Q}$.
Thus,
the three cases above are determined by the value assumed by the
quadratic form $\eta$ on  $a_p$, that is by
\[a_p^T M a_p=\rho^2 p^TS^TM^{-T}M M^{-1} Sp=-\rho^2 p^TSM^{-1}Sp.\]
The result now follows.
\end{proof}

For any $p\in R$ we have $Sp=0$; if
$p\in V_{\lambda}$, where $V_{\lambda}$ is an eigenspace
of $M^{-1}S$,  then $p^TS^TM^{-1}Sp=\lambda p^TS^Tp=0$.
In particular, the coordinates of all the points of $\fP$ (see Section~\ref{sketch}) satisfy
the system
\begin{equation}
\label{eqs} \begin{cases}
  p^TS^TM^{-1}Sp=0 \\
  p^TMp=0.
  \end{cases}
\end{equation}

\begin{lemma}
Suppose $\varphi$ to be an alternating form with a maximum number of
totally isotropic lines which are also totally singular for the
quadric $Q$. 
Then $\varphi$ can be represented with
respect to the basis $B$ by an antisymmetric matrix of the form
\begin{equation}
\label{smat}
 S=\begin{pmatrix}
  S_{11} & U      & \bZ &  \bZ \\
  -U^T & S_{22} & \bZ &  \bZ \\
  \bZ    & \bZ      & \bZ &  \bZ \\
  \bZ    & \bZ      & \bZ &  \bZ \\
  \end{pmatrix},
\end{equation}
where $S_{11}=-S_{11}^T$ and $U$, $S_{22}$ are given by Table
\ref{Matrix S22}.
\end{lemma}
\begin{proof}
The generic matrix of an antisymmetric form with radical $R$ is of
the form \eqref{smat}, with $S_{11}$ and $S_{22}$ antisymmetric and
$U$ arbitrary.
By Lemma \ref{maxeig} and Assumption \ref{eigA}, if the number of totally isotropic lines which are also totally singular is maximum, then there are two maximal
subspaces of dimension $m$
contained in a complement $H_0$ of $R\cap R^{\perp_Q}$ in $R^{\perp_Q}$ which are both totally singular and totally
isotropic. Thus, we may take the first two blocks of
columns of $S_{22}$ as described in Table \ref{Matrix S22}.
\begin{table}
\centerline{  \begin{tabular}{c|c|c|c}
    Case & $S_{22}$ &  $\nu$ & Non-null entries in $U$ \\ \hline
    1, 2& $\begin{pmatrix} \bZ & I & \bZ \\ -I & \bZ & \bZ \\ \bZ & \bZ & 0
    \end{pmatrix}$
     &  $n-(r+d)/2$ & $1$ in the position $U_{1,2n+1-(r+d)}$ \\
    3 & $\begin{pmatrix} \bZ & I \\ -I & \bZ \end{pmatrix}$ & $n-\frac{1}{2}(r+d-1)$ &
      None \\
    4 & $\begin{pmatrix} \bZ & I & \bZ & \bZ \\
                        -I & \bZ & \bZ & \bZ \\
                        \bZ & \bZ & 0 & \alpha \\
                        \bZ & \bZ & -\alpha & 0 \\
                      \end{pmatrix}$ & $n-\frac{1}{2}(r+d+1)$ &
                      \begin{minipage}{4cm}
                        \begin{itemize}
                        \item if $\alpha=0$,
                        a non-singular $2\times 2$ minor $U_I$
                        contained
                      in the last two columns of $U$.
                      \item if $\alpha\neq 0$, possibly a
                      $2\times 2$ minor $U_I$ contained
                      in the last two columns of $U$.
                      \end{itemize}
                    \end{minipage}
                    \end{tabular} }
\caption{Structure of the matrix $S$}
\label{Matrix S22}
\end{table}
Observe that the linear transformation induced by
\[ D=\begin{pmatrix}
  Z & \bZ & \bZ & \bZ \\
  \bZ & I & \bZ & \bZ \\
  \bZ & \bZ & I & \bZ \\
  \bZ & \bZ & \bZ & Z^{-T}
  \end{pmatrix}, \]
  with $Z$ a non-singular $d\times d$ matrix, acts on $M$ and $S$ as
  follows
\[ D^TMD=M,\qquad D^TSD=\begin{pmatrix}
  Z^TS_{11}Z & Z^TU & \bZ & \bZ  \\
  -U^TZ & S_{22} & \bZ & \bZ \\
  \bZ & \bZ & \bZ & \bZ \\
  \bZ & \bZ & \bZ & \bZ
  \end{pmatrix}.
  \]
In particular, by a suitable choice of $Z$, the matrix $U$ can be
assumed to be of the form
\[ \begin{pmatrix}
   \bZ & I \\
   \bZ & \bZ
   \end{pmatrix}, \]
   where $I$ is either the $1\times 1$ or $2\times 2$ identity matrix,
   according as the case being considered is 1, 2 or 4.

\end{proof}

\subsection{First case}
\label{fc}

Throughout this section we shall write the coordinates of a generic point $p$
with respect to the basis $B$ given by \eqref{basis} as
\[p=(\mathbf{x_1}, \mathbf{z_1}, \mathbf{z_2},y,\mathbf{x_2}, \mathbf{y_2})\]
where
$y\in\FF_q$, $\mathbf{z_1}, \mathbf{z_2}\in\FF_{q}^{n-(r+d)/2}$, $\mathbf{x_1}, \mathbf{y_2}\in \FF_q^d$ and $\mathbf{x_2}\in\FF_q^{(r-d)}$. Furthermore,
$z\in\FF_q$ is taken to be the first component of the vector $\mathbf{x_1}$.
By Tables~\ref{table cases} and~\ref{table matrices}, we have
$R_0=R_0^+:=\begin{pmatrix} \bZ & I \\ I & \bZ\end{pmatrix}$
with $I$ the identity matrix
of order $(r-d)/2$. Then,
\[p^TS^TM^{-1}Sp=-z^2+2\mathbf{z_2}^T\mathbf{z_1}~~{\rm and}~~p^TMp=2\mathbf{z_2}^T\mathbf{z_1} +y^2 +\mathbf{x_2}^T R_0^+\mathbf{x_2}+2\mathbf{x_1}^T \mathbf{y_2}.\]

We need a preliminary technical lemma.
\begin{lemma}
\label{l11}
The following properties hold.
 \begin{enumerate}
\item
  \label{l11p1}
  For any given $\beta\in \FF_q\setminus \{0\}$, the number of solutions $(y,\mathbf{x_2})$ of the equation $y^2+\mathbf{x_2}^TR_0^+\mathbf{x_2}=\beta^2$ is   $q^{(r-d)/2}(q^{(r-d)/2}+1).$
\item\label{l11p2}
  Consider the quadratic form $\theta(z, \mathbf{z_1}, \mathbf{z_2})=-z^2+2\mathbf{z_2}^T\mathbf{z_1}.$ Then the number of vectors $(z, \mathbf{z_1}, \mathbf{z_2})$ with $z\not=0$ such that $\theta(z, \mathbf{z_1}, \mathbf{z_2})\in-\square$ is
  \begin{equation}
    \label{ccc}
   \sqval .
  \end{equation}
\end{enumerate}
\end{lemma}
\begin{proof}
\begin{enumerate}
\item
  Let $\xi(y,\mathbf{x_2})=y^2+\mathbf{x_2}^TR_0^+\mathbf{x_2}$ be a quadratic
  form defined on $J:=\langle b_{2n+1-r}, B_{D_0}\rangle$, where
  the coordinates of vectors are taken
  with respect to the basis $\{b_{2n+1-r}\}\cup B_{D_0}.$
  Then $\xi$ induces a parabolic quadric
  $R'$ of rank $(r-d)/2$ and
  the polar
  hyperplane of the vector
  $(1,0,\ldots,0)=(1,\mathbf{0})\in J$ cuts a hyperbolic
  section on $R'$. So, $(1,\mathbf{0})$ is external to $R'$ and,
 consequently, its orbit has size $\frac{1}{2}q^{(r-d)/2}(q^{(r-d)/2}+1)$.
 Furthermore,
 $\xi(1,\mathbf{0})=1\in\square$; thus the points external to
 $R'$ are always squares and they number to
 $\frac{1}{2}q^{(r-d)/2}(q^{(r-d)/2}+1)$. Note that for each square point $p$ there
 are exactly $2$ vectors $v_1,v_2$ representing $p=\langle v_1\rangle=
 \langle v_2\rangle$ such
 that $\xi(v_1)=\xi(v_2)=\beta^2$;
 thus,
 the overall number of solutions of the equation is
 $q^{(r-d)/2}(q^{(r-d)/2}+1)$.
\item
Consider the
quadratic form
$\theta(z, \mathbf{z_1}, \mathbf{z_2})=-z^2+2\mathbf{z_2}^T\mathbf{z_1}$
defined on the
space $J':=\langle b_1, B_{H_0}\setminus\{b_{2n+1-r}\}\rangle.$
Observe that the point $(1,0,\ldots,0)=(1,\mathbf{0})\in J'$
is always external to the parabolic
quadric of equation $\theta(z,\mathbf{z_1},\mathbf{z_2})=0$, as its
polar hyperplane cuts a hyperbolic section of equation
$2\mathbf{z_2}^T\mathbf{z_1}=0$.
Note that
$\theta(1,\mathbf{0})=-1$ --- namely, $\theta(1,\mathbf{0})$ is
a square if $-1\in\square$ and a non-square otherwise.
This gives that the number of vectors on which $\theta$ assumes
a square value when
$-1\in\square$ is the same as the number of values on which $\theta$ assumes
a non-square value for $-1\not\in\square$.
In particular, for $-1\in\square$,
the number of such vectors is the number of vectors
$(z, \mathbf{z_1}, \mathbf{z_2})$ corresponding to external points to
$\theta(z,\mathbf{z_1},\mathbf{z_2})=0$,
excluding those lying in the hyperplane $z=0$.
This gives \eqref{ccc}.
The same number  is obtained for $-1\not\in\square$.
\end{enumerate}
\end{proof}

\begin{proposition}
\label{c1p1}
Suppose we have a form $\varphi$ yielding the
maximum possible number of totally singular and totally isotropic
lines in Case 1. Then,
  \[
  A=2\frac{q^{n-(r+d)/2}-1}{q-1}
  +\frac{q^{r-1}-1}{q-1}+q^{(r+d-2)/2};\qquad
  N^0=\frac{q^{2n-1}-1}{q-1}-A; \]
  \[ N^+= \frac{1}{2}\left({q^{2n-1}+q^{2n-(r+d)/2-1}+ q^{n+(r+d)/2-1}-q^{n-1}}\right); \]
\[
N^-=\frac{1}{2}\left({q^{2n-1}-q^{2n-(r+d)/2-1}- q^{n+(r+d)/2-1}+q^{n-1}}\right). \]
\end{proposition}
\begin{proof}

By Assumption~\ref{eigA}, in order for the number
of totally singular, totally isotropic vectors to be maximum
we need $M^{-1}S$ to have two eigenspaces $V_{\lambda}, V_{\mu}$
of non-zero eigenvalues $\lambda,
\mu=-\lambda$, both of maximal dimension $m=n-(r+d)/2$; thus,
$A_V=2\frac{q^{n-(r+d)/2}-1}{q-1}$.
 As the quadric $Q_R$ induced by $\eta$ on $\PG(R)$ can be seen as
the product of a hyperbolic quadric of rank $(r-d)/2$ with
a subspace of dimension
$d$, we have \[
A_R=\#\{\text{points of } Q_R\}=\frac{q^{r-1}-1}{q-1}+q^{(r+d-2)/2}.
\]
It is now straightforward to retrieve $A.$

For any $p\in Q$, we have $p\in\fP^0$ if and only if $p\not\in\fP$
and the coordinates of $p$  are solution of System~\eqref{eqs}, that is
\begin{equation}\label{system N_0}
\left\{ \begin{array}{l}
 -z^2+2\mathbf{z_2}^T\mathbf{z_1}=0\\
 \\
 2\mathbf{z_2}^T\mathbf{z_1} +y^2 +\mathbf{x_2}^T R_0^+\mathbf{x_2}+2\mathbf{x_1}^T \mathbf{y_2}=0.\\
 \end{array}\right.
\end{equation}
To determine the number of solutions of \eqref{system N_0} we
distinguish three cases:
\begin{itemize}
\item
  $\mathbf{x_1}=\mathbf{0}$;
  consequently we also have $z=0$. Under this assumption
  the first equation in \eqref{system N_0} is $\mathbf{z_2}^T\mathbf{z_1}=0$; it has
  \[ (q^{n-(r+d)/2}-1)(q^{n-(r+d)/2-1}+1) +1 \]
  solutions in  $(\mathbf{z_1},\mathbf{z_2})$, that is the number of
  singular \emph{vectors} for the hyperbolic quadratic form
  $\mathbf{z_2}^T\mathbf{z_1}$ of rank $n-(r+d)/2$.
  Given $\mathbf{z_1}$ and $\mathbf{z_2}$, we can choose $\mathbf{y_2}$ in an
  arbitrary way; thus it can assume $q^d$ values.
  Finally, the second equation in \eqref{system N_0} is fulfilled
  when the vector $(y,\mathbf{x_2})$ is singular for the parabolic form
  $y^2+\mathbf{x_2}^TR_0^+\mathbf{x_2}$ of rank $(r-d)/2$; that is to say there are
  $q^{r-d}$ possibilities for it.
\par
Thus,
the number of (projective) points whose coordinates satisfy
\eqref{system N_0} with $\mathbf{x_1}=\mathbf{0}$ is
\[ N^0_1=\frac{q^r(q^{2n-r-d-1}+q^{n-(r+d)/2}-q^{n-(r+d)/2-1})-1}{q-1}.\]
\item
Assume now $z=0$ and $\mathbf{x_1}\neq\mathbf{0}.$
The first equation in \eqref{system N_0} is the same as before;
thus
the vector $(\mathbf{z_1},\mathbf{z_2})$ can assume
$(q^{n-(r+d)/2}-1)(q^{n-(r+d)/2-1}+1) +1$ distinct values.
As $z=0$ and $\mathbf{x_1}\neq\mathbf{0}$, the vector
$\mathbf{x_1}$ can be chosen in $q^{d-1}-1$ ways, while $y$ and $\mathbf{x_2}$
are arbitrary --- thus  there are respectively
$q$ and $q^{r-d}$ possibilities for these.
Observe that given $\mathbf{z_1},\mathbf{z_2},y,\mathbf{x_1},\mathbf{x_2}$
the second
equation in \eqref{system N_0} is a non-null linear equation in
$\mathbf{y_2}$; thus there are $q^{d-1}$ possible solutions $\mathbf{y_2}$.
Overall we get that
 the number of projective points satisfying~\eqref{system N_0} with  $z=0$ and
 $\mathbf{x_1}\neq\mathbf{0}$  is
\[N^0_2=\frac{q^r}{q-1}(q^{d-1}-1)(q^{2n-r-d-1}+q^{n-(r+d)/2}-q^{n-(r+d)/2-1}).\]
\item
Finally, suppose $z\neq0$. Clearly, there are $(q-1)$ possible
choices for $z$ and, consequently, $q^{d-1}(q-1)$ choices for $\mathbf{x_1}$.
The first equation in \eqref{system N_0} becomes
$\mathbf{z_2}^T\mathbf{z_1}=z^2$.
Observe that the hyperbolic form $\mathbf{z_2}^T\mathbf{z_1}$
assumes a given square value $z^2$ for exactly
$(\frac{q^{2n-r-d}-1}{q-1} -\frac{(q^{n-(r+d)/2}-1)(q^{n-(r+d)/2-1}+1)}{q-1})\frac{1}{2}\cdot 2$
choices of $(\mathbf{z_1},\mathbf{z_2})$.
The values of $y$ and $\mathbf{x_2}$ can now be chosen arbitrarily; that is
to say, there
are respectively $q$ and $q^{r-d}$ possibilities.
Finally, as in the previous case,
the vector $\mathbf{y_2}$ must be the solution of a non-null
linear equation in $d$ unknowns;
thus, it can assume $q^{d-1}$ distinct values.
So, for  $z\not=0$, then the number of projective points being solutions of
\eqref{system N_0}  is
\[N^0_3=q^{2n-2}-q^{n+(r+d)/2 -2}.\]
\end{itemize}
In particular,
\begin{equation}
N^0=N^0_1+N^0_2+N^0_3-A=\frac{q^{2n-1}-1}{q-1}-A.
\end{equation}
\par
By Lemma \ref{lintext} and Table \ref{intext},
$p\in\fP^+$ if and only if $p^TSM^{-1}Sp\in-\square$.
Thus, the coordinates of the points of $\fP^+$ satisfy
\begin{equation}\label{system N_+}
 \left\{ \begin{array}{l}
 -z^2+2\mathbf{z_2}^T\mathbf{z_1}\in -\square\\
 \\
 2\mathbf{z_2}^T\mathbf{z_1} +y^2 +\mathbf{x_2}^T R_0^+\mathbf{x_2}+2\mathbf{x_1}^T \mathbf{y_2}=0\\
 \end{array}\right.
\end{equation}
We argue as above.
\begin{itemize}
\item
Suppose $\mathbf{x_1}=\mathbf{0}$; hence $z=0.$
The vector $\mathbf{y_2}$ can be chosen arbitrarily; thus,
it may assume $q^d$ values.
The first equation in \eqref{system N_+} gives
$2\mathbf{z_2}^T\mathbf{z_1}=-\beta^2$ for some
element $\beta\in \FF_q\setminus \{0\}$.
As the quadric induced by $2\mathbf{z_2}^T\mathbf{z_1}$ is hyperbolic,
 there are
\[ (\frac{q^{2n-r-d}-1}{q-1} -\frac{(q^{n-(r+d)/2}-1)(q^{n-(r+d)/2-1}+1)}{q-1})\frac{1}{2}\cdot (q-1) \]
possible vectors of $V(2n-r-d,q)$ on which the quadratic form $2\mathbf{z_2}^T\mathbf{z_1}$ assumes a value opposite of a square.
Observe that for any of these choices of $(\mathbf{z_1},\mathbf{z_2})$, an
element $-\beta^2$ is determined by the first equation.
Given such $\beta^2$, the second equation becomes
\[ \beta^2=y^2+\mathbf{x_2}^TR_0^+\mathbf{x_2}. \]
By Claim 1 of Lemma \ref{l11}, the number of solutions of this equation is
$q^{(r-d)/2}(q^{(r-d)/2}+1)$.
So, the contribution of this case to the
number of points fulfilling   \eqref{system N_+}  is
\[N^+_1=
\frac{(q^{2n-r-d-1}-q^{n-(r+d)/2-1})(q^r+q^{(r+d)/2})}{2}.\]
\item
Suppose now $\mathbf{x_1}\neq\mathbf{0}$ and $z=0$.
The vector $\mathbf{x_1}$, clearly, can assume $q^{d-1}-1$ distinct non-null values.
The analysis of the first equation in \eqref{system N_+} is
exactly as before and gives that
the vector $(\mathbf{z_1},\mathbf{z_2})$ can assume
$(\frac{q^{2n-r-d}-1}{q-1} -\frac{(q^{n-(r+d)/2}-1)(q^{n-(r+d)/2-1}+1)}{q-1})\frac{1}{2}\cdot (q-1)$ values.
The values of $y$ and $\mathbf{x_2}$ may be assigned arbitrarily,
thus there are respectively
$q$ and $q^{r-d}$ possibilities for them.
Finally,
$\mathbf{y_2}$ can assume $q^{d-1}$ different values,
this being the number of solutions of a linear equation in $d$ unknowns.
So, for  $z=0$ and $\mathbf{x_1}\neq\mathbf{0}$,
the number of projective points solution of \eqref{system N_+}
is
\[N^+_2=\frac{(q^{2n-r-d-1}-q^{n-(r+d)/2-1})(q^{d-1}-1)q^{r}}{2}.\]
\item
Assume now $z\neq 0$.
By Lemma \ref{l11}, Claim \ref{l11p2},
$-z^2+2\mathbf{z_2}^T\mathbf{z_1}\in-\square$ has
\[
\sqval
 \]
solutions in $(z,\mathbf{z_1}, \mathbf{z_2})$. For $\mathbf{x_1}$
 there remain $q^{d-1}$ possibilities since the first coordinate $z$ has already been taken into account in the first equation.
 As before, the values of $y$ and $\mathbf{x_2}$ can be assigned arbitrarily,
that is in respectively $q$ and $q^{r-d}$ ways and
$\mathbf{y_2}$ is a solution of a linear equation in $d$ unknowns;
thus there are $q^{d-1}$ possibilities for the latter.
The contribution of solutions in terms of projective points to
the system \eqref{system N_+} for  $z\neq0$  is
\[N^+_3=\frac{q^{2n-1}+ q^{n+(r+d)/2-1}-q^{2n-2}+q^{n+(r+d)/2-2}}{2}.\]
\end{itemize}
In particular,
\begin{equation}
N^+={N^+_1+N^+_2+N^+_3}=\frac{q^{2n-1}+q^{2n-(r+d)/2-1}+ q^{n+(r+d)/2-1}-q^{n-1}}{2}.
\end{equation}
The value of $N^-$ can now be recovered either
with a similar argument, or by just observing that
$N^-=\#\{\text{points of } Q(2n,q)\}-(A+N^{0}+N^{+})$.
\end{proof}
\begin{proposition}
\label{max1}
The function
$f^1(r,d)$ attains its maximum $f^1_{\max}$
for $r=2n-1$ and $d=1$, where it assumes the
value
\[ f^1_{\max}:=f^1(2n-1,1)=
\frac{(q^{n-1}-1)(q^{3n-2}+q^{3n-3}-q^{3n-4}+q^{2n}-q^{n-1}-1)}{(q-1)^2(q+1)}.
 \]
\end{proposition}
\begin{proof}
By plugging the values of $A,N^0,N^+$ and $N^-$ into \eqref{key}, we get
\begin{multline}
f^1(r,d)(q+1)(q-1)^2=
q^{2n-3}(q-1)(q^{r-1}+q^{(r+d)/2+1}-q^{(r+d)/2-1}+q^{n-(r+d)/2+1}+q^{n-(r+d)/2})+
\\
q^{4n-3}-q^{2n-1}-q^{2n-2}+q^{2n-3}-q^{2n}+1.
\end{multline}
Recall that from the last paragraph of Section
\ref{sketch}, $1\leq d\leq \min\{r,2n-r\}$.

Let $s=r+d$.
It is
straightforward to see that
$f(r,d)$ is maximum if and only if $g(r,s)$ is maximum, where
\begin{equation}
\label{tilde g}
g(r,s)=q^{n-s/2+1}+q^{n-s/2}+q^{s/2+1}-q^{s/2-1}+q^{r-1},
\end{equation}
with the constraints $1\leq r\leq 2n-1$ and
\begin{equation}\label{bound x}
  r+1\leq s\leq\min\{2r,2n\}.
\end{equation}

In order to determine the maximum of ${g}(r,s)$ in its
domain, we regard it as a continuous function defined
over $\mathbb{R}^2$ and then we reinterpret its behavior
over $\mathbb{Z}^2$. So, we can consider the
derivative
\begin{equation}
\label{diff x}
 \frac{\partial}{\partial s}{g}(r,s)=\frac{\log q}{2}\left(
q^{s/2+1}-q^{n-s/2+1} -q^{s/2-1}-q^{n-s/2}\right).
\end{equation}
This is positive for
\[ q^{s/2}\left(q-\frac{1}{q}\right)>q^{n-s/2}(q+1), \]
that is
\begin{equation}
\label{ex>n}
 s>n+\log_q\frac{q}{q-1}.
\end{equation}
As $1<\frac{q}{q-1}< q$, also $0<\log_q\frac{q}{q-1}<1$,
and \eqref{ex>n} gives  $s\geq n+1$.
Hence, for $s\geq n+1$ the function ${g}(r,s)$ is increasing in $s$, while
for $s\leq n$ it is decreasing.
Define
\[ h(r):=\max_s g(r,s), \]
where $s$ varies in all allowable ways for any given $r$.
The following cases are possible:
\begin{itemize}
  \item
    for $s\geq n+1$ the maximum of $g(r,s)$
    is attained when $s$ is maximum, that is
    \[ h(r)=g(r,\max s), \]
    where by $\max s$ we denote the maximum value $s$ may assume, subject to
    the constraints of \eqref{bound x}.
    This leads to the following
    two subcases:
    \begin{enumerate}
      \item if $r>n$, then $\max s=2n$ and
        \[ h(r):=\max_s {g}(r,s)=g(r,2n). \]
        By \eqref{bound x},
        $r< s\leq 2n$ is odd and by
        \begin{equation}
          \label{diff r}
          \frac{\partial}{\partial r}h(r)=\frac{\partial}{\partial r}{g}(r,2n)=q^{r-1}\log q> 0,
        \end{equation}
        the value of $h(r)={g}(r,2n)$ is maximum
        for $r$  maximum, that is $r=2n-1$.  Since $s=r+d$ by definition,
        as $r=2n-1$ and $s=2n$, we have $d=1$.
        So, $g(r,s)$ is maximum for $r=2n-1$ and $s=2n$.
        The  value assumed in this case is
\begin{equation}
\label{proper bound}
 {g}(2n-1,2n)=q^{2n-2}+q^{n+1}-q^{n-1}+q+1.
\end{equation}
\item
   if $r\leq n$ and  also $s\geq n+1$, then
   $n+1\leq s\leq 2r$
   implies $r\geq (n+1)/2.$
  Since $s\leq 2r\leq 2n$, by~\eqref{bound x}, we have $\max s=2r$. Thus,
   \[ h(r)=g(r,2r)=q^{r+1}+q^{n-r+1}+q^{n-r}. \]
   Then,
   \[ \frac{\partial}{\partial r}h(r)=(\log q)\left(q^{r+1}-q^{n-r+1}-q^{n-r}
     \right). \]
   We have $\frac{\partial}{\partial r}h(r)>0$ if and only if
    $(r+1)>(n-r)+\log_q(q+1)$,
   that is $2r>n-1+\log_q(q+1)$, i.e. $2r\geq n+1.$ In particular, for $2r\geq n+1$,
   the function $h(r)$ is
   increasing and it attains its maximum for $r=n$, where
   \[ h(n)=q^{n+1}+q+1=g(n,2n). \]
   This is smaller than \eqref{proper bound}; so in the range
   $\frac{n+1}{2}\leq r\leq 2n-1$ the maximum is given by \eqref{proper bound}.
\end{enumerate}
\item
  Suppose now $2\leq s< n+1$; then, as $d\geq 1$, we have
  $r\leq s-1$ and
  \begin{equation}
    \label{est2} \begin{cases}
    r-1\leq s-2 < n-1 \\
    n-s/2+1\leq n \\
    n-s/2\leq n-1 \\
    s/2+1\leq n/2+3/2 \\
    s/2-1\geq 0.
  \end{cases}.
\end{equation}
Using the estimates of \eqref{est2} in \eqref{tilde g} we get
\begin{equation}
\label{in range}
{g}(r,s) < q^{n}+2q^{n-1}+q^{n/2+3/2}=:g_0.
\end{equation}
Observe that the value $g_0$ from \eqref{in range} is always smaller
than that of $g(r,s)$ given by \eqref{proper bound}.
\end{itemize}
The above argument proves that the maximum of ${g}(r,s)$ is always attained in
\eqref{proper bound};
consequently, the maximum for $f^1(r,d)$ is  $f_{\max}^1:=f^1(2n-1,1)$.
\end{proof}
In particular,
\[\frac{q^{2n}-1}{q-1}\frac{q^{2n-2}-1}{q-1}\frac{1}{q+1}-f_{\max}^1=
q^{4n-5}-q^{3n-4}.\]
Thus, this case correspond to words of minimum weight and these
words are alternating bilinear forms with a radical of dimension $2n-1$,
that is to say, they correspond to some points of $\cG_{2n+1,2n-1}$.

We now show that in the three remaining cases $f(r,d)$ cannot ever
be larger than $f_{\max}^1$.

\subsection{Cases 2 and 3}
\label{rc}
Cases 2 and 3 can be carried out in close analogy to
Section~\ref{fc}. The values they yield for
$f^2_{\max}$ and
$f^3_{\max}$
turn out to be always lower than $f^1_{\max}$. The following
proposition summarizes the results; its proof is quite
analogous  to that of Proposition
\ref{c1p1}.
\begin{proposition}
\label{c234p1}
Suppose we have a maximum number of totally singular totally isotropic
lines and Assumption \ref{eigA} holds. Then,
\begin{itemize}
\item
In case 2,
\[  A=2\frac{q^{n-(r+d)/2}-1}{q-1}
  +\frac{q^{r-1}-1}{q-1}-q^{(r+d-2)/2};\qquad N_0=\frac{q^{2n-1}-1}{q-1}-A;
\]
\[ N^+=\frac{q^{2n-1}+q^{2n-(r+d)/2-1}-q^{n+(r+d)/2-1}-q^{n-1}}{2}; \]
\[ N^-=\frac{q^{2n-1}-q^{2n-(r+d)/2-1}+q^{n+(r+d)/2-1}+q^{n-1}}{2}.
 \]
\item
 In case 3,
\[ A=2\frac{q^{n-(r+d-1)/2}-1}{q-1}+\frac{q^{r-1}-1}{q-1};
\qquad N_0=\frac{q^{2n-1}+q^{n+(r+d-1)/2}-q^{n+(r+d-1)/2-1}-1}{q-1}-A; \]
\[ N^{\pm}=\frac{q^{2n-r-d}-q^{n-(r+d+1)/2}}{2(q-1)}(q^{r+d-1}\pm q^{(r+d-1)/2}). \]
\end{itemize}
\end{proposition}
Using arguments similar to those of Proposition \ref{max1}, we can
show that for $n>3$ the maximum number of totally singular totally
isotropic lines is attained when the radical $R$
of the alternating form $\varphi$ is as large as possible.
Our results are described by the following proposition.
\begin{proposition}
For $1\leq i\leq 3$ denote by $f^i(r,d)$ the function $f(r,d)$
obtained in case $i$ and
by $f^i_{\max}$ its maximum. The values of $r,d$ where $f^i$ attains
its maximum  $f^i_{\max}$ for $i=1,2,3$ are those outlined in
Table
\ref{max table}.
\begin{table}[h]
\centerline{
\begin{tabular}{c|cc|cc}
  Case & \multicolumn{2}{c|}{$n=3$} & \multicolumn{2}{|c}{$n>3$} \\ \hline
  1    & $r=5$, & $d=1$ & $r=2n-1$,& $d=1$         \\
  2    & $r=1$, & $d=1$ & $r=2n-1$,& $d=1$ \\
  3    & $r=1$, & $d=0$ & $r=2n-1$,& $d=0$ \\
\end{tabular} }
\caption{Maxima for the functions $f^i(r,d)$.}
\label{max table}
\end{table}
\end{proposition}

\subsection{Fourth case}
\label{4c}
As $Q_0$ is elliptic, we have from Table \ref{Matrix S22}
\[ A=2\frac{q^{n-(r+d+1)/2}-1}{q-1}+\frac{q^{r-1}-1}{q-1}. \]
Recall that $A^0>B^+>B^0>B^-$; thus,
\[ f^4_{\max}<\frac{1}{q+1}\left(AA_0+(\frac{q^{2n}-1}{q-1}-A)B^+\right)=
\frac{A_0-B^+}{q+1}A+\frac{q^{2n}-1}{q^2-1}B^+. \]
Some straightforward algebraic manipulations show that
\[
  f^4_{\max}(r,d)(q-1)^2(q+1)<\tau(r,s),
\]
where $s=r+d$ and
\begin{multline*}
\tau(r,s):=
q^n(q^{n-1}-1)(q-1)(q^{r-3}+2q^{n-(s+5)/2})+ \\
q^{4n-3}+q^{3n-1}-q^{3n-2}-3q^{2n-2}+2q^{2n-3}-q^{2n}+2q^{n-1}-2q^{n-2}+1.
\end{multline*}
Regarding $\tau (r,s)$ as a function defined over $\mathbb{R}^2$,
\[ \frac{\partial}{\partial s}\tau(r,s)=\log q\left(
  -q^{3n-(s+5)/2}+q^{3n-(s+7)/2}+q^{2n-(s+3)/2}-q^{2n-(s+5)/2}\right)<0.
\]
In particular,
the maximum of $\tau(r,s)$ is attained for $s$ minimum, that is $s=r$.
Thus,
\[ h(r):=\max_s \tau(r,s)=\tau(r,r). \]
By computing $\frac{\partial}{\partial r}h(r)$,
we see that the function $h(r)$ has one critical point in the range
$1<r<2n-1$ and this critical point is a minimum. Thus, the maximum
of $h(r)$ is  for either $r=1$ or $r=2n-1$.
We have
\[ h(1)=q^{4n-3}+q^{3n-1}-q^{3n-2}+2q^{3n-3}-2q^{3n-4}-4q^{2n-2}+
3q^{2n-3}-q^{2n}+q^{n-1}-q^{n-2}+1, \]
\[ h(2n-1)=q^{4n-3}+q^{4n-4}-q^{4n-5}+q^{3n-1}-q^{3n-2}-q^{3n-3}+q^{3n-4}
-q^{2n-2}-q^{2n}+1, \]
and $h(1)<h(2n-1)$.
In any case, $f^4_{\max}<\frac{h(2n-1)}{(q-1)^2(q+1)}=f^1_{\max}$.
\begin{proposition}
\label{mp0}
  Under Assumption \ref{eigA},
  the maximum of the function $f(r,d)$ is $f^1_{\max}$, attained in case 1,
  for $r=2n-1$ and $d=1$; consequently, the minimum distance
  of the orthogonal Grassmann code $\cP_{n,2}$ is
  $q^{4n-5}-q^{3n-4}$.
\end{proposition}

\subsection{Removal of Assumption \ref{eigA}}
\label{ga}
We are now ready to drop Assumption \ref{eigA}.
Let $\cW$ be the quadric induced by the matrix $W=SM^{-1}S$, as studied in
Lemma \ref{lintext}.
We say that a line $\ell\in\Delta_{n,2}$ is of type $0$, $+$,
$\alpha$, $\beta$ or
$-$ according to the conditions in Table \ref{tell}; observe that
actually $\ell\cap \cW=\ell\cap(\fP\cup\fP^0)$.
\begin{table}
  \[ \begin{array}{c|ccc|c}
    \text{Type of $\ell$} & \#(\ell\cap\fP^+) & \#(\ell\cap \cW) &
    \#(\ell\cap\fP^-) & \text{Corresponding set} \\[.2em] \hline
       0                  &   0 & q+1 & 0 & \Delta^0 \\
       +                  &   q & 1 & 0 & \Delta^+ \\
       \alpha             &   \frac{q+1}{2} & 0 & \frac{q+1}{2} &
       \Delta^{\alpha} \\
       \beta              &   \frac{q-1}{2} & 2 & \frac{q-1}{2} &
       \Delta^{\beta} \\
       -                  &   0 & 1 & q & \Delta^- \\
       \end{array} \]
       \caption{Types of lines in $\Delta_{n,2}$ and corresponding subsets}
       \label{tell}
\end{table}
\begin{lemma}
\label{ldel}
  For any choice of $M$ and $S$ we have
  \[ (N^+-N^-)\frac{q^{2n-2}-1}{q-1}=q(\#\Delta^+-\#\Delta^-). \]
\end{lemma}
\begin{proof}
  We count the number of flags of type $(p,\ell)$ with $p\in\fP^-$ or
  $p\in\fP^+$ and $\ell\in\Delta_{2,n}$ in two different ways.
  Let
  \[ \mathfrak S^-=\{ (p,\ell) : p\in\fP^-, p\in\ell, \ell\in\Delta_{n,2} \}. \]
  As there are exactly $\frac{q^{2n-2}-1}{q-1}$ lines of $\Delta_{2,n}$ through
  any point $p\in Q$ we have
  \[ \#\mathfrak S^-=N^-\frac{q^{2n-2}-1}{q-1}. \]
  On the other hand, only lines of type $\alpha$, $\beta$ or $-$ are incident
  with points of $\fP^-$. Using Table \ref{tell} we get
  \[ \#\mathfrak S^-= q\#\Delta^-+\frac{q+1}{2}\#\Delta^{\alpha}+
  \frac{q-1}{2}\#\Delta^{\beta}. \]

So, \[N^-\frac{q^{2n-2}-1}{q-1}=q\#\Delta^-+\frac{q+1}{2}\#\Delta^{\alpha}+
  \frac{q-1}{2}\#\Delta^{\beta}.\]

  By the same counting argument on $\mathfrak S^+:=\{ (p,\ell) : p\in\fP^+, p\in\ell, \ell\in\Delta_{n,2} \}$, we have
\[N^+\frac{q^{2n-2}-1}{q-1}=q\#\Delta^++\frac{q+1}{2}\#\Delta^{\alpha}+
  \frac{q-1}{2}\#\Delta^{\beta}.\]


  Consequently,
  \[ (N^+-N^-)\frac{q^{2n-2}-1}{q-1}=
  q(\#\Delta^+-\#\Delta^-). \]
\end{proof}
\begin{proposition}
\label{mpp}
The maximum of the function $f$ is $f^1_{\max}$, as
described in Proposition \ref{mp0}.
\end{proposition}
\begin{proof}
Both the lines of $\Delta^+$ and those of $\Delta^-$ are simultaneously
tangent to $Q$ and $1$--secant to $\cW$. In particular, all of them
are tangent to both $Q$ and $\cW$ at some point $p\in Q\cap\cW$.
Thus,
\[ \#\Delta^++\#\Delta^-\leq \frac{q^{2n-2}-1}{q-1}\#(Q\cap \cW). \]
\emph{A fortiori},
\[ \#\Delta^+-\#\Delta^-\leq \frac{q^{2n-2}-1}{q-1}\#(Q\cap \cW); \]
consequently, by Lemma \ref{ldel},
\begin{equation}
\label{edel} \delta:=(N^+-N)\leq q\#(Q\cap \cW)< q\frac{q^{2n}-1}{q-1}.
\end{equation}
Since $B^++B^-=2B^0$, $B^+-B^0=q^{n-2}$,
$A^0-B^0=q^{2n-3}$, we can rewrite \eqref{key} in the form
\begin{equation}
 (q+1)f=A(A^0-B^0)+B^0\frac{q^{2n}-1}{q-1}+\delta(B^+-B^0)=
Aq^{2n-3}+\delta q^{n-2}+\frac{q^{2n-3}-1}{q-1}\frac{q^{2n}-1}{q-1}.
\end{equation}
Using the estimate of \eqref{edel} this becomes
\begin{equation}\label{ebound} f<\frac{Aq^{2n-3}(q-1)^2+q^{4n-3}-q^{3n-1}-q^{2n-3}+q^{3n}-q^{2n}-q^{n}+q^{n-1}+1}{(q+1)(q-1)^2}.
\end{equation}
Suppose now $f\geq f^1_{\max}=\frac{q^{4n-3}+q^{4n-4}+\ldots}{(q+1)(q-1)^2}$.
Then, $Aq^{2n-3}(q-1)^2\geq q^{4n-4}$; thus, $A\approx q^{2n-3}$.
Observe now that $A=A_V+A_R$ with $A_V\leq2\frac{q^n-1}{q-1}$. In particular,
we have $A_R\approx q^{2n-3}$. As $A_R=\#(Q\cap R)$ this gives
$r=\dim R\geq 2n-1$.
As $R\leq D^{\perp_Q}$, we obtain $d\leq 2$. For $d=1$  there are no
possible eigenspaces contributing to $A_V$ and we end up in either case 1 or
2, according to the nature of $R_0$; thus, $f=f_{\max}^1$.
Likewise, for $d=2$ there are also no possible
eigenspaces contributing to $A_V$ and we are done.
Finally, for $d=0$ 
then $M^{-1}S$ has necessarily two eigenspaces of dimension $1$, that is
of maximal dimension; thus we end up in either case 3
or 4. In particular, all possible configurations have been already investigated
and we can conclude
$f=f_{\max}^1$.
\end{proof}

\subsection{Minimum weight codewords}
\label{pe}

\begin{proposition}
\label{mw}
  All minimum weight codewords are projectively equivalent.
\end{proposition}
\begin{proof}
By Proposition \ref{mpp} a minimum weight codeword corresponds to a
configuration in which $\dim R=2n-1$, $d=1$ and
$\dim V_{\lambda}=\dim V_{\mu}=0$.  The form of the
matrices $M$ and $S$ with respect to the basis $B$ of Section
\ref{fMS}
is as dictated by Case 1.
In particular, we have with respect to $B$
\[ M=\begin{pmatrix}
  0 & 0 & \bZ    & 1 \\
  0 & 1 & \bZ     &  0 \\
  \bZ & \bZ & R_0^+ & \bZ \\
  1 & 0 & \bZ     & 0 \\
  \end{pmatrix},\qquad
  S=\begin{pmatrix}
    0 & 1 & \bZ & 0 \\
    -1 & 0   & \bZ & 0 \\
    \bZ & \bZ   & \bZ & \bZ \\
    0 & 0   & \bZ & 0
    \end{pmatrix}. \]
This means that, up to projective equivalence, $M$ and $S$
are uniquely determined. The result follows.
\end{proof}

\vskip.2cm\noindent
\begin{minipage}[t]{\textwidth}
Authors' addresses:
\vskip.2cm\noindent\nobreak
\centerline{
\begin{minipage}[t]{7cm}
Ilaria Cardinali\\
Department of Information Engineering and Mathematics\\University of Siena\\
Via Roma 56, I-53100, Siena, Italy\\
ilaria.cardinali@unisi.it\\
\end{minipage}\hfill
\begin{minipage}[t]{6cm}
Luca Giuzzi\\
D.I.C.A.T.A.M. \\ Section of Mathematics \\
Universit\`a di Brescia\\
Via Branze 53, I-25123, Brescia, Italy \\
luca.giuzzi@unibs.it
\end{minipage}}
\leftline{
\begin{minipage}[t]{7cm}
Antonio Pasini \\
Department of Information Engineering and Mathematics\\University of Siena\\
Via Roma 56, I-53100, Siena, Italy\\
antonio.pasini@unisi.it\\
\end{minipage}}
\end{minipage}


\begin{thebibliography}{999}

\bibitem{IL13} I. Cardinali and L. Giuzzi, \emph{Codes and caps from
orthogonal Grassmannians}, {Finite Fields Appl.} {\bf 24} (2013),
148-169.

\bibitem{CP2} I. Cardinali and A. Pasini, \emph{Grassmann and Weyl
  Embeddings of Orthogonal Grassmannians}, J. Algebr. Comb., {\bfseries 38}
(2013), 863-888.
\bibitem{CP1} I. Cardinali and A. Pasini, \emph{Veronesean embeddings of
  dual polar spaces of orthogonal type}, J. Combin. Theory Ser. A.
{\bfseries 120} (2013), 1328-1350.

\bibitem{GK2013} S.R. Ghorpade, K. V. Kaipa,
  \emph{Automorphism groups of Grassmann codes},
  {Finite Fields Appl.} {\bf 23} (2013),
  80-102.


\bibitem{GL2001} S.R. Ghorpade, G. Lachaud, \emph{Hyperplane sections of
    Grassmannians and the number of MDS linear codes}, {Finite
    Fields Appl.} {\bf 7} (2001), 468-506.

\bibitem{GPP2009} S.R. Ghorpade, A.R. Patir, H.K. Pillai,
  \emph{Decomposable subspaces, linear sections of Grassmann varieties, and
    Higher weights of Grassmann codes}, {Finite Fields Appl.}
  {\bf 15} (2009), 54-68.

\bibitem{HT91} J. W. P. Hirshfeld and J. A. Thas,  \emph{General Galois
  geometries},  Clarendon Press, Oxford (1991).
\bibitem{L} G. Lachaud,
  \emph{Number of points of plane sections and linear codes defined on
  algebraic varieties} in {Arithmetic, Geometry and Coding Theory,
  Luminy, France 1993}, De Gruyter (1996), 77-104.
\bibitem{MS} F.J. MacWilliams and N.J.A. Sloane, \emph{The theory of
    error correcting codes}, North--Holland Publishing Co.,
  Amsterdam--New York--Oxford (1977).
\bibitem{N96} D. Yu. Nogin, \emph{Codes associated to Grassmannians}  in
  { Arithmetic, geometry and coding theory (Luminy, 1993)},
  de Gruyter (1996),  145–154.
\bibitem{ALA} S. Roman, \emph{Advanced linear algebra}, GTM 135, Springer
  Verlag (2005).
\bibitem{R1} C.T. Ryan, \emph{An application of Grassmannian varieties to
  coding theory}, {Congr. Numer.} {\bf 57} (1987), 257-271.
\bibitem{R2} C.T. Ryan, \emph{Projective codes based on Grassmann varieties},
  {Congr. Numer.} {\bf 57} (1987), 273-279.
\bibitem{TVN} M.A. Tsfasman, S.G. Vl{\u{a}}du{\c{t}}, D.Yu. Nogin,
  \emph{Algebraic geometric codes: basic notions},
  Mathematical Surveys and Monographs {\bf 139},
  American Mathematical Society (2007).

\end{thebibliography}
\end{document}